\documentclass[12pt,a4]{article}

\usepackage{amsmath,graphicx,amsthm,amssymb,mathrsfs} 

\newcommand{\exn}{{\bf E}}

\newcommand{\R}{\mathbb{R}}

\newcommand{\be}{\beta}

\newcommand{\ga}{\gamma}
\newcommand{\si}{\sigma}

\newcommand{\De}{\Delta}
\newcommand{\f}{\varphi}

\newcommand{\deq}{\stackrel{d}{=}}
\newcommand{\ep}{\varepsilon}

\newcommand{\ind}{{\bf 1}}

\newtheorem{theo}{Theorem}

\newtheorem{lemo}{Lemma}

\title{A note on recovering the Brownian motion component from a L\'evy process}
\author{Konstantin Borovkov$^1$ }

\date{}

\begin{document}
	\maketitle

\footnotetext[1]{School of Mathematics and Statistics, The University of Melbourne, Parkville 3010, Australia; e-mail: borovkov@unimelb.edu.au.}

\begin{abstract}
Gonz\'alez C\'azares  and  Ivanovs~(2021) suggested a new method for ``recovering'' the Brownian motion component from the trajectory of a L\'evy process that required sampling from an independent Brownian motion process. We show that such a procedure works equally well without any additional source of randomness if one uses normal quantiles  instead of the ordered increments of the auxiliary Brownian motion process.

		
	\end{abstract}

\section{Introduction and main results}

The present note is complementing  recent interesting paper~\cite{GoIv21} (see also related paper~\cite{FoGoIv21}) presenting an original idea on how to ``recover'' the Brownian motion component from the observed L\'evy process 
\[
X_t=Y_t + \si W_t, \quad t\in [0,1],
\]
where the standard Brownian motion $W$ is independent of the pure jump process~$Y$ and $\si>0.$ More precisely, the authors of~\cite{GoIv21} recovered the path of the Brownian bridge $\{W_t - t W_1\}_{t\in [0,1]}$ (noting that it is not possible to consistently ``extract'' the linear drift from the Brownian motion process trajectory, due to the equivalence of the distributions of Brownian motion processes with different linear drifts). 

Apart from being an interesting mathematical result by itself, such a separation can be useful in some applications as well. For instance, if the Brownian component is interpreted as noise, it enables one to recover (up to a linear drift) the signal~$Y$ from the observed process~$X$. Further comments (and some relevant references) on how one can benefit from such a  separation in statistical problems can be found in~\cite{GoIv21}.

The first method used in~\cite{GoIv21} for recovering the Brownian motion component   was based on a  construction that curiously  required randomization. To describe that method, we need to introduce some notations. For an $n$-tuple  $\mathcal X=(x_1, \ldots, x_n)$ of  real numbers that are all different from each other  (this will a.s.\ be the case for all the random samples considered in this note, so without loss of generality we will assume in what follows that they all do have this property, omitting ``a.s.'' in the respective relations), denote by  $R_k (\mathcal X):= \sum_{j=1}^n \ind (x_j\le x_k)$  the rank of $x_k$ in~$\mathcal X$, $k= 1,\ldots, n$, and by  $(\mathcal X)_j$ and $[  \mathcal X]_j$ the  $j$th component of $\mathcal X$ and $j$th order statistic for $  \mathcal X$, respectively, so that $(\mathcal X)_{k}=x_k $ and $[\mathcal X]_{R_k (\mathcal X)}=x_k $. For a random process $\{V_t\}_{t\in [0,1]}$ and $n\ge 1,$ we set
$ \De_{n } V:=(V_{1/n}-V_{0}, V_{2/n}-V_{1/n},\ldots ,V_{1 }-V_{(n-1)/n})\in  \R^n. $

Now assume that $\{W'_t\}_{t\in [0.1]}$ is a standard Brownian motion process that is independent of~$X$ and put 
\[
W^{(n)}_t:= \sum_{i=1}^{\lfloor nt\rfloor } [\De_n  W' ]_{R_i (\De_n X)}, \quad  t\in [0,1].
\]
In words, we first re-order one-step  increments of $W'$ on the grid $k/n,$ $k=0,1,\ldots, n,$ such that the sequence of their ranks is the same as the sequence of the ranks of the increments of~$X$ on the same grid, and then form $W^{(n)}$ as the process of the partial sums of that re-ordered sequence of the increments of $W'.$ 

The following theorem re-states the first part  of the main result in~\cite{GoIv21}. Set $\be^*:=\inf\big\{p>0:\int_{(-1,1)}|x|^p\Pi (dx)<\infty \big\},$ where $\Pi$ is the jump measure of~$Y$. 

\begin{theo}
	 \label{T1}
For any $p\in (\be^*,2]\cup \{2\},$ as $n\to\infty,$  
\begin{align}
\label{ConvT1}
\sup_{t\in [0,1]} \big|W_t - W^{(n)}_t - (W_1 - W^{(n)}_1)t \big|  =o_P (n^{-1/2+p/4}) . 
\end{align}
\end{theo}
 
The natural question that arises here is about the role of the   independently sampled process~$W'$: why do we need this auxiliary random object? Does it necessarily need to be a standard Brownian motion? It it possible to modify the suggested method to avoid using any auxiliary independent random processes? 

Simulations showed that the described scheme still works when  $W'$ is an independent  fractional Brownian motion process with an arbitrary Hurst parameter $H\in (0,1)$ (one just needs to scale the process so that the marginal distributions of the components of $\De_n W'$ would be the same as for $\De_n W$). This observation suggested that the true role of the auxiliary process $W'$ is just to provide approximations to the normal quantiles and that one can ``recover'' $W$ using this kind of approach without sampling any independent random process. We show in the present note  that this is the case indeed. 

Denote by $\Phi$ the standard normal distribution function, by $\overline{\Phi}  :=1 - \Phi   $ the distribution tail of~$\Phi,$ by $\f$ the density of~$\Phi$, and by $Q:=\Phi^{-1}$ the standard normal quantile function. For $ u_{n,k}:=\frac{k}{n+1},$ $ n\ge 1,$ $1\le k\le n,$ set 
\begin{align}
\label{TildeW}
  \widetilde W^{(n)}_t:= n^{-1/2}\sum_{i=1}^{\lfloor nt\rfloor }Q(u_{n,R_i (\De_n X)}), \quad  t\in [0,1].
 \end{align}
 
Our main result is the following theorem. 
\begin{theo}
 \label{T2}
The assertion of Theorem~\ref{T1} remains true if the process  $W^{(n)}$ is replaced in it with~$\widetilde W^{(n)}.$ 
\end{theo}

\section{Proofs}

\begin{proof}[Proof of Theorem~\ref{T2}] As in the proof of Theorem~\ref{T1}  in~\cite{GoIv21}, we will start  with the  observation that 
\begin{align}
\label{grid}
\sup_{t\in [0,1]} |W_t - W_{\lfloor nt\rfloor/n}|
 =O_P \big(n^{-1/2}(\ln n)^{1/2}\big),
\end{align}
which is an immediate consequence of the L\'evy's modulus of continuity theorem. 
Therefore, in the problem of bounding the error in the version of~\eqref{ConvT1} with $\widetilde W^{(n)},$ we only need to consider the maximum of the absolute deviations  on the grid $t= \frac{i}n,$ $1\le i\le n .$

To this end, we observe  that $\widetilde W^{(n)}_1=0$ since $Q(\frac12+h)+Q(\frac12-h)=0,$ $h\in [0,\frac12),$    and hence, letting
\begin{align*}
\eta_{n,k}:  =(\De_n W)_k- n^{-1/2} Q(u_{n,R_k (\De_n X)}) -n^{-1}W_1,
\end{align*}
one has 
\begin{align}
\label{WW}
W_{i/n} -\widetilde W^{(n)}_{i/n} - (W_{1} -\widetilde W^{(n)}_{1} )i/n
 = \sum_{k=1}^ i \eta_{n,k}.
\end{align}
Next, similarly to the decomposition of $\xi_{ni}$ on p.\,2422 in~\cite{GoIv21}, we write $\eta_{n,k}= \widetilde\eta_{n,k}+\widehat \eta_{n,k},$ where
\begin{align*}
\widetilde \eta_{n,k}:& = [\De_n W]_{ R_k (\De_n X)} - n^{-1/2} Q(u_{n,R_k (\De_n X)}) -n^{-1}W_1,
\\
\widehat  \eta_{n,k}:& =(\De_n W)_k - [\De_n W]_{  R_k (\De_n X)}
\end{align*}
That 
\begin{align}
\label{hat}
\max_{1\le i\le n} \bigg|\sum_{k=1}^i \widehat\eta_{n,k}\bigg| = o_P (n^{-1/2+p/4})  \quad\mbox{as \ } n\to\infty 
\end{align} 
was proved on p.\,2425 in~\cite{GoIv21}. To complete the proof of our theorem, we will now show that
\begin{align}
\label{tilde}
\max_{1\le i\le n} \bigg|\sum_{k=1}^i \widetilde \eta_{n,k}\bigg|
= O_P \big(  n^{-1/2 }(\ln\ln n)^{1/2 }\big)  \quad\mbox{as \ } n\to\infty.
\end{align}
First note that it is not hard to verify that~\eqref{tilde} is equivalent to the assertion that, for any positive sequence $\ep_n\to 0,$ 
\begin{align}
\label{tildes}
\max_{1\le i\le n} \bigg|\sum_{k=1}^i \widetilde \eta_{n,k}\bigg|
  = o_P \big(\ep_n^{-1} n^{-1/2 }(\ln\ln n)^{1/2 }\big)  \quad\mbox{as \ } n\to\infty.
\end{align} 
Further, it is easy to see that the random variables $ \widetilde \eta_{n,1}, \ldots,  \widetilde \eta_{n,n}$ are  exchangeable and $\sum_{k=1}^n \widetilde \eta_{n,k} =0$. Therefore, setting $\gamma_{n,k}:=\ep_n n^{ 1/2 }(\ln\ln n)^{-1/2 }\widetilde \eta_{n,k}, $ $k=1,\ldots, n,$  the desired relation~\eqref{tildes} (and hence~\eqref{tilde}) immediately follows from Lemmata~\ref{L00} and~\ref{L0} below, the latter implying that $\sum_{k=1}^n\gamma_{n,k}^2 =O_P(\ep_n^2)=o_P(1).$  Now the assertion of Theorem~\ref{T2} follows from  representation~\eqref{WW} and relations~\eqref{grid}, \eqref{hat} and~\eqref{tilde}.\end{proof}  

\begin{lemo}
	\label{L00}
	Let $\ga_{n,1}, \ga_{n,2},\ldots,   \ga_{n,n}$, $n\ge 1,$ be a triangular array of random variables that are exchangeable in each row and such that $\sum_{k=1}^n \ga_{n,k} =0$ a.s.\ and, 
	as $n\to \infty,$
\begin{align}
\label{Squares}
\sum_{k=1}^n \ga_{n,k}^2 =  o_P (1).
\end{align}
Then $\max_{1\le i \le n} \big|\sum_{k=1}^i \ga_{n,k}\big|
    = o_P  (1).$
\end{lemo}

The assertion of Lemma~\ref{L00} is an immediate consequence of Theorem~3.13 in~\cite{Ka05} on convergence of partial sums processes. In our case,  the limiting process in that theorem is identically equal to zero, the convergence of characteristic triples following from the assumptions and the obvious observation that $ \max_{1\le k\le n}  |  \ga_{n,k} |=  o_P (1)$ from~\eqref{Squares}.

\begin{lemo}
	\label{L0}
As $n\to \infty,$
	\[
	\sum_{k=1}^n\widetilde  \eta_{n,k}^2 = O_P \big(n^{-1 }  \ln \ln n \big). 
	\]
\end{lemo}

\begin{proof}[Proof of Lemma~\ref{L0}]

Note that the components of the vector 
	\[
	\mathcal Z_n = (Z_1, \ldots, Z_n):=  n^{1/2} \bigl((\De_n W)_1, \ldots, (\De_n W)_n \bigr)
	\]
are   independent standard normal random variables, and let $\overline{Z}_n:=n^{-1}\sum_{k=1}^n Z_k .$

Setting $\zeta_{n,k}:= [\mathcal Z_n]_{k} -  Q(u_{n,k}),$ $k=1,\ldots, n,$ 
we observe that $\overline{\zeta}_n:=n^{-1}\sum_{k=1}^n \zeta_{n,k} =\overline{Z}_n$ and hence 
\[ \widetilde \eta_{n,k} 
 \equiv  n^{-1/2} ( [\mathcal Z_n]_{R_k (\De_n X)} -  Q(u_{n,R_k (\De_n X)}) - \overline{Z}_n)
 = n^{-1/2} (\zeta_{n,R_k (\De_n X)} -    \overline{\zeta}_n).  
\]
Therefore, 
\[
\sum_{k=1}^n \widetilde \eta_{n,k}^2 
 =  n^{-1} \sum_{k=1}^n (\zeta_{n.k}- \overline{\zeta}_n)^2
  = n^{-1} \sum_{k=1}^n  \zeta_{n.k}^2 - \overline{\zeta}_n^2.
\]
Here $\overline{\zeta}_n^2=\overline{Z}_n^2 \deq n^{-1} Z_1^2 = O_P (n^{-1}). $ Further, denoting by 
\[
Q^*_n(u):=\sum_{k=1}^n [\mathcal Z_n]_k 
 \ind (u\in \mbox{$[\frac{k-1 }n, \frac{k}n)$}),  \quad u\in (0,1), 
\]
the empirical quantile function for $\mathcal Z_n $   and letting
\[
Q_n(u):=\sum_{k=1}^n Q(u_{n,k}) 
 \ind (u\in \mbox{$[\frac{k-1 }n, \frac{k}n)$}),  \quad u\in (0,1),
\] 
one has
\begin{align*}
 n^{-1} \sum_{k=1}^n  \zeta_{n.k}^2
  &=  n^{-1} \sum_{k=1}^n ([\mathcal Z_n]_k- Q(u_{n,k}))^2
   =\int_0^1 (Q^*_n(u) - Q_n (u))^2 du
   \\
   & \le 2 \int_0^1 (Q^*_n(u) - Q  (u))^2 du + 2 \int_0^1 (Q (u) - Q_n (u))^2 du.
\end{align*}
It follows from Theorem~1 in~\cite{BeFo20} that the first term in the second  line    is $O_P (n^{-1}\ln \ln n),$ whereas the second  term  in that line is $O(n^{-1})$ by Lemma~\ref{L1} below. This completes the proof of Lemma~\ref{L0}. 
\end{proof}

\begin{lemo}
 	\label{L1}
 For any $n\ge 1 ,$ one has 
 	\[
 	\int_0^1 (Q(u) - Q_n (u))^2 du \le 3.73 n^{-1} . 
 	\]
\end{lemo}

The proof of this bound uses the following three elementary auxiliary results.

\begin{lemo}
	\label{L2}
	Let $f\in C^1 (a_0, b_0)$ be convex and non-decreasing on $[a,b]\subset (a_0, b_0)$. Then, for any $v_0\in [a,b],$ 
\[
 \int_a^b (f(v)-f(v_0))^2 dv  \le \mbox{$\frac13$} (f'(b))^2 (b-a)^3.
\]
\end{lemo} 
\begin{proof}[Proof of Lemma~\ref{L2}] For $v,v_0\in [a,b],$ one has 
\[
| f(v)-f(v_0)| = \biggl|\int_{v_0}^v f'(s)\, ds \biggr| \le   f'(b) |v-v_0| 
\]	
as $f'(s) \ge 0$ is non-decreasing by assumption. Hence
\begin{align*}
\int_a^b (f(v)-f(v_0))^2 dv  \le (f'(b))^2\int_a^b (v-v_0)^2 dv 
 \le  \mbox{$\frac13$} (f'(b))^2 \big((b-v_0)^3-  (a-v_0)^3),
\end{align*}
completing the proof since clearly $y^3 - x^3 \le (y-x)^3$ for $x\le 0\le y.$\end{proof}

 \begin{lemo}
 	\label{L3}
 	For any  $u\in (0.5, 1),$  one has $1-u \le \sqrt{\pi/2} \f (Q(u)). $ 
 \end{lemo}

\begin{proof}[Proof of Lemma~\ref{L3}] The desired inequality follows from the observation that it turns into equality at the endpoints $u=0.5$ and $u= 1$ and that its RHS is a concave function as $[\f (Q(u))]'' = -\sqrt{2\pi}e^{Q^2(u)/2}<0.$  \end{proof}

\begin{lemo}
	\label{L4}
Assume that $f(v),$ $ v\in [a,b],$ is a non-decreasing function, $v_0\in [a,\frac12 (a+b)].$ Then
\[
I_1: = \int_a^b (f(v)-f(v_0))^2 dv 
 \le  \int_a^b (f(v)-f(a))^2 dv=:I_2. 
\]
\end{lemo}

\begin{proof}[Proof of Lemma~\ref{L4}] Expanding the squares, one has
\begin{align*}
I_2- I_1 &=
 2 (f(v_0)- f(a)) \int_a^b f(v) \, dv + (f^2 (a) - f^2 (v_0))(b-a)
 \\
 & =  (f(v_0)- f(a))\biggl[ 2 \int_a^b f(v) \, dv -(f(v_0)+ f(a))(b-a) \biggr]\ge 0
\end{align*}	
since, due to the monotonicity of~$f$, 
\begin{align*}
\int_a^b f(v) \, dv &=\int_a^{v_0 }\cdots + \int_{v_0 }^b\cdots
\ge 
 f(a)(v_0- a)  +   f  (v_0) (b-v_0)
\\
& =\mbox{$ \frac12 $} f(a)(b- a)  +\mbox{$ \frac12 $} f(v_0)(b- a)  
+  (f(v_0)- f(a)) (\mbox{$ \frac12 $}  (a+b) - v_0),  
\end{align*}
where the last term is non-negative by the assumptions. 	
\end{proof}

\begin{proof}[Proof of Lemma~\ref{L1}] Putting $q_{n,k}:=Q(u_{n,k}),$ $k=1, \ldots, n,$ one has 
\begin{align}
\label{SumQ}
	\int_0^1 (Q(u) - Q_n (u))^2 du  
	=\sum_{k=1}^n  
	  \int_{ (k-1)/n }^{k/n} (Q(u) -q_{n,k})^2 du .
\end{align}

By symmetry, it is enough to bound the terms with $k\ge n/2$ only, assuming for simplicity that $n$ is even. 

As $Q$ is clearly convex and increasing on $(\frac12,1),$ for $n/2<k<n$ we get  by Lemmata~\ref{L2} and~\ref{L3} that 
\begin{align*}
 \int_{ (k-1)/n }^{k/n} (Q(u) -q_{n,k})^2 du
  &\le
   \mbox{$\frac13$} (Q'(k/n)) ^2 n^{-3}
  =\frac{n^{-3}}{3\f^2 (Q(k/n))} \le \frac{ \pi  n^{-1}}{6(n-k)^2}.
\end{align*}
Therefore 
\begin{align}
\label{InnerSum}
	 \sum_{k=n/2+1} ^{n-1} 
\int_{ (k-1)/n }^{k/n} (Q(u) -q_{n,k})^2 du
\le
 \frac{ \pi}{ 6 n }\sum_{k=n/2+1} ^{n-1} \frac{1}{ (n-k)^2}
 \le
 \frac{ \pi}{ 6 n }\sum_{m=1} ^{\infty} \frac{1}{ m^2}=\frac{\pi^3}{36 n}. 
\end{align}

For the last term in the sum on the RHS of~\eqref{SumQ}, setting $q:= Q(1-1/n),$ from Lemma~\ref{L4} we obtain that 
\begin{align*}
J_n:&=  \int_{ (n-1)/n }^{1} (Q(u) -Q(u_{n,n}) )^2 du
\le  \int_{ (n-1)/n }^{1} (Q(u) -q )^2 du
\\
& 
= \int_{ (n-1)/n }^{1}  Q(u) ^2 du -2 q \int_{ (n-1)/n }^{1}  Q(u) du
+q^2 n^{-1}.
\end{align*}
Integrating by parts, we get 
\begin{align*}
  \int_{ (n-1)/n }^{1}  Q(u) ^2 du 
  & =\exn (Z_1^2; Z_1>q)
   = \int_q^\infty  z^2\f (z) dz 
   \\
   & = [-z\f(z)]_q^\infty + \int_q^\infty \f (z) dz 
   = q \f (q) +\overline{\Phi}(q),  
\end{align*}
whereas  
\begin{align*}
\int_{ (n-1)/n }^{1}  Q(u)  du 
& =\exn (Z_1 ; Z_1>q)
= \int_q^\infty  z \f (z) dz 
= - \int_q^\infty d\f (z) 
=  \f (q) .
\end{align*}
Since $\overline{\Phi}(q)=n^{-1} $ and $ q^2 n^{-1}-q \f (q)
= q^2 (\overline{\Phi}(q) - \f(q)/q)<0 $ by the well-known inequality for
the normal  Mills' ratio (see e.g.\ Ch.~VII.1 in~\cite{Fe68}), we conclude that 
\begin{align*}
J_n\le    q \f (q) + n^{-1} - 2  q \f (q)+q^2 n^{-1} 
  \le  n^{-1}. 
\end{align*}
Together with~\eqref{InnerSum} and an elementary bound for the constant $\frac{\pi^3}{18}+2 <3.73$  this completes the proof of Lemma~\ref{L1}. \end{proof}

\end{document}